\documentclass[10pt]{article}
\usepackage{graphicx,epsfig,amsmath,amsthm,amssymb}
\usepackage{enumerate}
\usepackage{multirow}
\usepackage{bbm}
\usepackage{balance}
\usepackage{footnote}
\usepackage{tablefootnote}
\usepackage{xr}

\newcommand{\norm}[1]{\left\lVert #1 \right\rVert}

\newcommand{\argmax}{\operatornamewithlimits{argmax}}

\newtheorem{lemma}{Lemma}

\newtheorem{theorem}{Theorem}

\begin{document}

\title{On the Fenchel Duality between Strong Convexity and Lipschitz Continuous Gradient}  %
\author{Xingyu Zhou\\Department of ECE\\The Ohio State University\\zhou.2055@osu.edu}

\maketitle
\begin{abstract}

We provide a simple proof for the Fenchel duality between strong convexity and Lipschitz continuous gradient. To this end, we first establish equivalent conditions of convexity for a general function that may not be differentiable. By utilizing these equivalent conditions, we can directly obtain equivalent conditions for strong convexity and Lipschitz continuous gradient. Based on these results, we can easily prove Fenchel duality. Beside this main result, we also identify several conditions that are implied by strong convexity or Lipschitz continuous gradient, but are not necessarily equivalent to them. This means that these conditions are more general than strong convexity or Lipschitz continuous gradient themselves.

\end{abstract}



\section{Introduction}

Fenchel duality between strong convexity and Lipschitz continuous gradient was first proved in \cite{hiriart2013convex}. Roughly speaking, it says that under mild conditions, (i) if $f$ is strongly convex with parameter $\mu$, then its conjugate $f^*$ has a Lipschitz continuous gradient with parameter ${1}{/\mu}$; (ii) if $f$ has a Lipschitz continuous gradient with parameter $L$, then its conjugate $f^*$ is strongly convex with parameter ${1}{/L}$. The original proof is quite long and thus hard to obtain the key insights. In this paper, we provide a simple proof of this important result. The key ingredients in our proof are the equivalent conditions for strong convexity and Lipschitz continuous gradient. To obtain these equivalent conditions, we first show that the traditional equivalent conditions for a smooth function to be convex still hold if we replace the gradient by subgradient for a non-smooth function. With this result, we are then able to directly identify equivalent conditions for strongly convex and Lipschitz continuous gradient. Beside these equivalent conditions, we also derive several conditions that are implied by strongly convex (or Lipschitz continuous gradient), but are not necessarily equivalent to them. This means that these conditions are more general than strongly convex (or Lipschitz continuous gradient). In addition, through this process, we demonstrate some common tricks in proving equivalence in optimization.

\section{Main Result}
The following theorem states our main result, i.e., Fenchel duality between strong convexity and Lipschitz continuous gradient.

\begin{theorem}
\label{thm:main}
	\begin{enumerate}[(i)]
	A function $f$ and its Fenchel conjugate function $f^*$ satisfy the following assertions:
		\item If $f$ is closed and strong convex with parameter $\mu$, then $f^*$ has a Lipschitz continuous gradient with parameter $\frac{1}{\mu}$.
		\item If $f$ is convex and has a Lipschitz continuous gradient with parameter $L$, then $f^*$ is strong convex with parameter $\frac{1}{L}$.
	\end{enumerate}
\end{theorem}

In order to present a simple proof for this result, we will present some useful results first. 

\section{Equivalent Conditions for Convexity}
In this section, we generalize the equivalent conditions of convexity for a smooth function to the general case that the function may be non-smooth. We show that once we replace the gradient by the subgradient, all the previous conditions still hold for a general function. 

\begin{lemma}[Equivalence for Convexity]
		Suppose $f: \mathbb{R}^n \rightarrow \mathbb{R}$ with the extended-value extension. Then, the following statements are equivalent:
		\begin{enumerate}[(i)]
			\item {(Jensen's inequality):} $f$ is convex.
			\item {(First-order):} $f(y) \ge f(x) + s_x^T(y-x)$ for all $x$, $y$ and any $s_x \in \partial f(x)$.
			\item {(Monotonicity of subgradient):} $(s_y - s_x)^T(y-x)\ge 0$ for all $x$, $y$ and any $s_x \in \partial f(x)$, $s_y \in \partial f(y)$.
		\end{enumerate}	
\end{lemma}

\begin{proof}
	(i) $\implies$ (ii): Since $f$ is convex, by definition, we have for any $0 < \alpha \le 1$
	\begin{align*}
		&f(x + \alpha(y-x)) \le \alpha f(y) + (1-\alpha)f(x)\\
		\iff& \frac{f(x + \alpha(y-x)) - f(x)}{\alpha} \le f(y) - f(x)\\
		\overset{(a)}{\implies} & f'(x,y-x) \le f(y) - f(x)\\
		\overset{(b)}{\implies} & s_x^T(y-x) \le f(y) - f(x) \quad \forall s_x \in \partial f(x)
	\end{align*}
	where (a) is obtained by letting $\alpha \to 0$ and the definition of \textit{directional derivate}. Note that the limit always exists for a convex $f$ by the  monotonicity of difference quotient; (b) follows from the fact that $f'(x,v) = \sup_{s \in \partial f(x)} s^Tv$ for a convex $f$; see \cite{boyd2004convex}.

	(ii) $\implies$ (iii): From (ii), we have 
	\begin{align*}
		&f(y) \ge f(x) + s_x^T(y-x)\\
		&f(x) \ge f(y) + s_y^T(x-y)
	\end{align*}
	Adding together directly yields (iii).

	(iii) $\implies$ (i): To this end, we will show (iii) $\implies$ (ii) and (ii) $\implies$ (i). 
	
	First, (iii) $\implies$ (ii):  Let $\phi(\alpha):= f(x + \alpha(y-x))$ and $x_{\alpha} := x + \alpha(y-x)$, then
	\begin{align}
		f(y) - f(x) = \phi(1) - \phi(0) = \int_0^1 s_{\alpha}^T(y-x) d\alpha
	\end{align}
	where $s_{\alpha} \in \partial f(x_{\alpha})$ for $t \in [0,1]$. Then by (iii), for any $s_x \in \partial f(x)$, we have 
	\begin{align*}
		s_{\alpha}^T(y-x) \ge s_x^T(y-x)
	\end{align*}
	which combined with inequality above directly implies (ii).

	Second, (ii) $\implies$ (i): By (ii), we have 
	\begin{align}
		&f(y) \ge f(x_{\alpha}) + s_t^T(y-x_{\alpha}) \label{eq:1}\\
		&f(x) \ge f(x_{\alpha}) + s_t^T(x-x_{\alpha}) \label{eq:2}
	\end{align}
	Multiplying \eqref{eq:1} by $\alpha$ and \eqref{eq:2} by $1-\alpha$, and adding together yields (i).
\end{proof}

\section{Conditions Related to Strong Convexity}
In this section, we will show several conditions that are equivalent to strong convexity by adopting the equivalent conditions of convexity shown in the last section. Moreover, we are also able to identify several conditions that are more general than strong convexity in the sense that these conditions can be implied by strong convexity but the reverse direction does not hold. 

\begin{lemma}[Equivalence for Strong Convexity]
\label{lem:SC_equiv}
		Suppose $f: \mathbb{R}^n \rightarrow \mathbb{R}$ with the extended-value extension. Then, the following statements are equivalent:
		\begin{enumerate}[(i)]
			\item $f$ is strongly convex with parameter $\mu$.
			\item $f(\alpha x+ (1-\alpha) y) \le \alpha f(x) + (1-\alpha) f(y)-\frac{\mu}{2}\alpha(1-\alpha)\norm{y-x}^2$ for any $x$, $y$.
			\item $f(y) \ge f(x) + s_x^T(y-x) + \frac{\mu}{2}\norm{y-x}^2$ for all $x$, $y$ and any $s_x \in \partial f(x)$.
			\item $(s_y - s_x)^T(y-x)\ge \mu \norm{y-x}^2$ for all $x$, $y$ and any $s_x \in \partial f(x)$, $s_y \in \partial f(y)$.
		\end{enumerate}	
	\end{lemma}

\begin{proof}
	(i) $\iff$ (ii): It follows from the definition of convexity for $g(x) = f(x)-\frac{\mu}{2}x^Tx$.

(i) $\iff$ (iii): It follows from the equivalent first-order condition of convexity for $g$.

(i) $\iff$ (iv): It follows from the equivalent monotonicity of (sub)-gradient of convexity for $g$.
\end{proof}

\begin{lemma}[Implications of Strong Convexity (SC)]
\label{lem:SC_implication}
	Suppose $f: \mathbb{R}^n \rightarrow \mathbb{R}$ with the extended-value extension. The following conditions are all implied by strong convexity with parameter $\mu$:
	\begin{enumerate}[(i)]
			\item $\frac{1}{2}\norm{s_x}^2\ge \mu (f(x)-f^*),~\forall x$ and $s_x \in \partial f(x)$.
			\item $\norm{s_y-s_x} \ge \mu \norm{y - x}$ $\forall x$, $y$ and any $s_x \in \partial f(x)$, $s_y \in \partial f(y)$.
			\item $f(y) \le f(x) + s_x^T(y-x) + \frac{1}{2\mu}\norm{s_y-s_x}^2$ $\forall x$, $y$ and any $s_x \in \partial f(x)$, $s_y \in \partial f(y)$.
			\item $(s_y - s_x)^T(y-x)\le \frac{1}{\mu} \norm{s_y-s_x}^2$  $\forall$$x$, $y$ and any $s_x \in \partial f(x)$, $s_y \in \partial f(y)$.
	\end{enumerate}	
\end{lemma}
\textbf{Note:} In the case of smooth function $f$, (i) reduces to {Polyak-Lojasiewicz (PL) inequality}, which is more general than strong convexity and is extremely useful for linear convergence.

\begin{proof}
	(SC) $\implies$ (i): By equivalence of strong convexity in the previous lemma, we have 
\begin{align*}
	f(y) \ge f(x) + s_x^T(y-x) + \frac{\mu}{2}\norm{y-x}^2
\end{align*}
Taking minimization with respect to $y$ on both sides, yields
\begin{align*}
	f^* \ge f(x) - \frac{1}{2\mu}\norm{s_x}^2
\end{align*}
Re-arranging it yields (i).

SC $\implies$ (ii): By equivalence of strong convexity in the previous lemma, we have 
\begin{align*}
	(s_y - s_x)^T(y-x)\ge \mu \norm{y-x}^2
\end{align*}
Applying Cauchy-Schwartz inequality to the left-hand-side, yields (ii).

SC $\implies$ (iii): For any given $s_x \in \partial f(x)$, let $\phi_x(z) := f(z) - s_x^T z$. First, since 
\begin{align*}
	(s_{z_1}^{\phi} - s_{z_2}^{\phi})(z_1 - z_2) = (s_{z_1}^{f} - s_{z_2}^{f})(z_1 - z_2) \ge \mu\norm{z_1 - z_2}^2
\end{align*}
which implies that $\phi_x(z)$ is also strong convexity with parameter $\mu$.

Then, applying PL-inequality to $\phi_x(z)$, yields
\begin{align*}
	\phi_x^*  = f(x) - s_x^Tx &\ge \phi_x(y) - \frac{1}{2\mu}\norm{s_y^{\phi}}^2\\
	& = f(y) - s_x^Ty - \frac{1}{2\mu}\norm{s_y - s_x}^2.
\end{align*}
Re-arranging it yields (iii).

SC $\implies$ (iv): From previous result, we have 
\begin{align*}
	&f(y) \le f(x) + s_x^T(y-x) + \frac{1}{2\mu}\norm{s_y-s_x}^2\\
	&f(x) \le f(y) + s_y^T(x-y) + \frac{1}{2\mu}\norm{s_x-s_y}^2
\end{align*}
Adding together yields (iv).
\end{proof}

\section{Conditions Related to Lipschitz Continuous Gradient}
In this section, we present a detailed relationship between all conditions that are related to Lipschitz continuous gradient. In particular, we show that if $f$ is convex, then all the following conditions are equivalent.

\begin{lemma}
\label{lem:LCG}
	For a function $f$ with a Lipschitz continuous gradient over $\mathbb{R}^n$, the following relations hold:
	\begin{align}
		[5] \iff [7] \implies [6] \implies [0] \implies [1] \iff [2] \iff [3] \iff [4] \label{eq:relation}
	\end{align}
If the function $f$ is convex, then all the conditions $[0]-[7]$ are equivalent.
\begin{align*}
	[0]~&\lVert\nabla f(x) - \nabla f(y)\rVert \le L \lVert x-y\rVert,~\forall x,y.\\
	[1]~&g(x) = \frac{L}{2}x^T x - f(x) \text{ is convex },~\forall x\\
	[2]~&f(y)\le f(x)+\nabla f(x)^T(y-x)+\frac{L}{2}\lVert y-x\rVert^2,~\forall x,y.\\
	[3]~&(\nabla f(x) - \nabla f(y)^T(x-y) \le L \rVert x-y\rVert^2, ~\forall x,y.\\
	[4]~&f(\alpha x+ (1-\alpha) y) \ge \alpha f(x) + (1-\alpha) f(y) - \frac{\alpha (1-\alpha)L}{2}\lVert x-y\rVert^2,~\forall x,y \text{ and } \alpha \in [0,1]\\
	[5]~&f(y)\ge f(x)+\nabla f(x)^T(y-x)+\frac{1}{2L}\lVert\nabla f(y)-\nabla f(x)\rVert^2,~\forall x,y.\\
	[6]~&(\nabla f(x) - \nabla f(y)^T(x-y) \ge \frac{1}{L} \lVert \nabla f(x)-\nabla f(y)\rVert^2, ~\forall x,y.\\
	[7]~&f(\alpha x+ (1-\alpha) y) \le \alpha f(x) + (1-\alpha) f(y) - \frac{\alpha (1-\alpha)}{2L}\lVert\nabla f(x)-\nabla f(y)\rVert^2,~\forall x,y \text{ and }\alpha \in [0,1].
\end{align*}
\end{lemma}

\begin{proof}

[1] $\iff$ [2]: It follows from the first-order equivalence of convexity.

[1] $\iff$ [3]: It follows from the monotonicity of gradient equivalence of convexity.

[1] $\iff$ [4]: It follows from the definition of convexity.

[5] $\implies$ [7]: Let $x_{\alpha} := \alpha x + (1-\alpha)y$, then 
\begin{align*}
	&f(x)\ge f(x_{\alpha})+\nabla f(x_{\alpha})^T(x-x_{\alpha})+\frac{1}{2L}\lVert \nabla f(x)-\nabla f(x_{\alpha})\rVert^2\\
	&f(y)\ge f(x_{\alpha})+\nabla f(z)^T(y-x_{\alpha})+\frac{1}{2L}||\nabla f(y)-\nabla f(x_{\alpha})||^2
\end{align*}
Multiplying the first inequality with $\alpha$ and second inequality with $1-\alpha$, and adding them together and using $\alpha \lVert x\rVert^2 + (1-\alpha) \lVert y\rVert^2 \ge \alpha (1-\alpha)\lVert x-y\rVert^2$, yields [7]. 

[7] $\implies$ [5]: Interchanging $x$ and $y$ in [7] and re-writing it as 
\begin{align*}
	f(y) \ge f(x) + \frac{f(x+\alpha (y-x)) -f(x)}{\alpha} + \frac{1-\alpha}{2L}\lVert\nabla f(x) - \nabla f(y)\rVert^2
\end{align*}
Letting $\alpha \to 0$, yields [5].

[5] $\implies$ [6]: Interchanging $x$ and $y$ and adding together.

[6] $\implies$ [0]: It follows directly from Cauchy-Schwartz inequality.

[0] $\implies$ [3]: It following directly from Cauchy-Schwartz inequality.

Now, suppose $f$ is convex, then we have 
[3] $\implies$ [5]: Let $\phi_x(z): = f(z) - \nabla f(x)^T z$. Note that since $f$ is convex, $\phi_x(z)$ attains its minimum at $z^* = x$. By [3], we have 
\begin{align*}
	(\nabla \phi_x(z_1) - \nabla \phi_x(z_2)^T(z_1-z_2) \le L \rVert z_1-z_2\rVert^2
\end{align*}
which implies that 
\begin{align*}
	\phi_x(z)\le \phi_x(y)+\nabla \phi_x(y)^T(z-y)+\frac{L}{2}\lVert z-y\rVert^2.
\end{align*}
Taking minimization over $z$, yields [5].
\end{proof}

\section{Useful Results On Fenchel Conjugate}
Before we finally prove the main result, we need the following useful results on conjugate functions.

\begin{lemma}
\label{lem:useful_1}
Consider the following conditions for a general function $f$:
\begin{align*}
	[1]~& f^*(s) =  s^Tx - f(x)\\
	[2]~& s \in \partial f(x)\\
	[3]~& x \in \partial f^*(s)
\end{align*}
Then, we have 
\begin{align*}
	[1] \iff [2] \implies [3]
\end{align*}
Further, if $f$ is closed and convex, then all these conditions are equivalent.
\end{lemma}

\begin{proof}

	[1] $\iff$ [2]: 
\begin{align*}
	s \in \partial f(x) &\iff f(y) \ge f(x) + s^T(y-x)~\forall y\\
	& \iff s^Tx - f(x) \ge  s^Ty -f(y)~\forall y\\
	& \iff s^Tx - f(x) \ge f^*(s)
\end{align*}
Since $f^*(s) \ge s^Tx - f(x)$ always holds, [1] $\iff$ [2].

[2] $\implies$ [3]: If [2] holds, then by previous result we have $f^*(s) = s^Tx - f(x)$. Thus,
\begin{align*}
	f^*(z) &= \sup_u(z^T u - f(u))\\
	& \ge z^T x - f(x)\\
	& = s^Tx - f(x) + x^T(z-s)\\
	& = f^*(s) + x^T(z-s)
\end{align*}
This holds for all $z$, which implies $x \in \partial f^*(s)$.

For a closed and convex function, we have 
[3] $\implies$ [2]: This follows from $f^{**} = f$. 
\end{proof}

\begin{lemma}[Differentiability]
	For a closed and strictly convex $f$, $\nabla f^*(s) = \argmax_x(s^T x - f(x))$.
\end{lemma}
\begin{proof}
	Suppose $x_1 \in \partial f^*(s)$ and $x_2 \in \partial f^*(s)$. By closeness and convexity, we have [3] $\implies$ [2], which gives 
	\begin{align*}
		s \in \partial f{x_1} \text{ and } s \in \partial f{x_2}
	\end{align*}
	Then, by previous result ([2] $\implies$ [1]), we have 
	\begin{align*}
		x_1 = \argmax(s^Tx -f(x)) \text{ and } x_2 = \argmax(s^Tx -f(x))
	\end{align*}
	By strictly convexity, $s^T x -f(x)$ has a unique maximizer for every $s$. Thus, $x:=x_1 = x_2$. Therefore, $\partial f^*(s) = \{x\}$, which implies $f^*$ is differentiable at $s$ and $\nabla f^*(s) = \argmax_x(s^T x - f(x))$.
\end{proof}

\section{Proof of Theorem \ref{thm:main}}
\label{sec:proof}
Now, we are well prepared to present the proof of our main result stated in Theorem \ref{thm:main}.
\begin{proof}
	(i): By implication of strong convexity in Lemma \ref{lem:SC_implication} , we have 
	\begin{align*}
		\norm{s_x - s_y} \ge \mu\norm{x-y}~\forall s_x\in \partial f(x), s_y\in \partial f(y)
	\end{align*}
	which implies (based on Lemma \ref{lem:useful_1})
	\begin{align*}
		\norm{s_x - s_y} \ge \mu\norm{\nabla f^*(s_x)- \nabla f^*(s_y)}
	\end{align*}
	Hence, $f^*$ has a Lipschitz continuous gradient with $\frac{1}{\mu}$ by definition.

	(2): By implication of Lipschitz continuous gradient for a convex $f$ shown in Lemma \ref{lem:LCG},  we have 
\begin{align*}
	(\nabla f(x) - \nabla f(y)^T(x-y) \ge \frac{1}{L} \lVert \nabla f(x)-\nabla f(y)\rVert^2
\end{align*}
which implies (based on Lemma \ref{lem:useful_1})
\begin{align*}
	(s_x - s_y)^T (x-y) \ge \frac{1}{L} \norm{s_x - s_y}^2~\forall x \in \partial f^*(s_x), y \in \partial f^*(s_y)
\end{align*}
Hence, $f^*$ is strongly convex with parameter $\frac{1}{L}$ based on Lemma \ref{lem:SC_equiv}.

\end{proof}

\section{Conclusion}
  A simple proof is presented for the Fenchel duality between strong convexity and Lipschitz continuous gradient. The key idea is based on the equivalence or implication of strong convexity and Lipschitz continuous gradient for a general function. We formally show the relationship between all the conditions that are related to strong convexity and Lipschitz continuous gradient, respectively. Moreover, through this process, we also demonstrate some useful tricks in proving equivalence in optimization, which may be utilized to prove more important results in the future.

\bibliographystyle{plain}
\bibliography{ref}

\end{document}